\documentclass[12pt,leqno]{amsart}
\usepackage[all]{xy}
\usepackage{amssymb,mathrsfs,euscript,enumitem}
\usepackage[backgroundcolor=yellow,shadow,colorinlistoftodos]{todonotes}
\usepackage{mathrsfs,euscript,amssymb,amsmath,bbold}
\usepackage{fourier}
\usepackage[usenames,dvipsnames]{pstricks}
\usepackage{lettrine}
\usepackage[hypertex]{hyperref}
\catcode`\@=11 \def\today{May 15, 2024}
\def\@evenfoot{\rule{0pt}{20pt}[\today] \hfill [{\tt \jobname.tex}]}
\def\@oddfoot{\rule{0pt}{20pt}{[\tt \jobname.tex}]\hfill [\today]}
\catcode`\@=13

\makeatletter
\providecommand\@dotsep{5}
\def\listtodoname{List of Todos}
\def\listoftodos{\@starttoc{tdo}\listtodoname}
\makeatother

\newtheorem{theorem}{Theorem}
\newtheorem{corollary}[theorem]{Corollary}

\newtheorem{lemma}[theorem]{Lemma}
\newtheorem{proposition}[theorem]{Proposition}

\theoremstyle{definition}
\newtheorem{example}[theorem]{Example}
\newtheorem{remark}[theorem]{Remark}
\newtheorem{definition}[theorem]{Definition}

\def\Lin{{\it Lin}}

\def\cmm{{canonical minimal model}}
\def\bbk{{\mathbb k}}

\def\di{d}
\def\znamenko#1{{(-1)^{#1}}\cdot}
\def\oaa#1#2{{\overline{a}_{#1,#2}}}

\def\dga{{(A,\di,\cdot\,)}}
\def\pp{\pi}
\def\ps{\psi}
\def\bfD{{\mathbb D}}
\def\aa#1#2{a_{#1,#2}}
\def\Ainfty{{$A_\infty$}}
\def\isot{semi\-canonical}
\def\isot{canonical}
\def\ot{\otimes}

\def\Ainfty{\hbox{${A_\infty}$}}

\def\id{{\rm id}}

\def\C{(A,\di,\cdot\  )}
\def\bfmu{{\boldsymbol \mu}}

\def\bfnu{{\boldsymbol \nu}}
\def\Rada#1#2#3{{#1_{#2},\ldots,#1_{#3}}}
\def\<{\langle}
\def\>{\rangle}
\def\bfmu{{\boldsymbol \mu}}
\def\rada#1#2{#1,\ldots,#2}

\begin{document}

\baselineskip15pt 

\title[Massey products revisited]{Strong minimal model theorem and Massey products}

\author{Martin Markl}
\address{The Czech Academy of Sciences, Institute of Mathematics, {\v Z}itn{\'a} 25,
         115 67 Prague, Czech Republic}
\email{markl@math.cas.cz}
\thanks{Supported by Praemium Academi\ae\ and RVO: 67985840.}


\begin{abstract}
Kadeishvili's minimal model theorem establishes
the existence of an
\Ainfty-structure,  unique up to isomorphism, 
on the cohomology of a dg
associative algebra, which captures its homotopy type. 
In this note we prove the existence of minimal models
that are unique up to isotopy, a stronger result
obviously known to T.~Kadeishvili
and certainly to others, yet seemingly overlooked by mankind.
We will explore how this stronger result can help in the study of Massey products.

First, we show that the attempts to extract a local information
from the ternary operation $\mu_3$ of our minimal model leads directly
to the rediscovery of the triple Massey product.  The motto~is:
\begin{center}
{\em
 The triple Massey product is an invariant manifestation of\/ $\mu_3$.
}
\end{center}
We then prove that, under reasonable
assumptions, the higher Massey product $\<\Rada x1n\>$ equals 
the set of all values $\mu_n(\Rada x1n)$, where $\mu_n$ runs over the 
$n$-ary products of our minimal models.  

We believe that this note
will help to elucidate 
the still somewhat enigmatic relationship between minimal models
and Massey products.
\end{abstract}

\maketitle

\tableofcontents

\section*{Introduction}

\lettrine{\color{red} T}{he}
standard minimal model theorem states that the homotopy type of a
differential graded associative algebra $A =(A,\di,\cdot)$ 
is captured by a suitable 
\Ainfty-structure $(H,0,\mu_2,\mu_3,\ldots)$ 
on its cohomology  $H = H(A,\di)$, and that this structure 
is unique up to an isomorphism. Although Tornike Kadeishvili in his
seminal paper~\cite{Kadei}, 
in which this theorem was formulated and proven, 
dropped the remark that the
minimal model he constructed  is, 
in today's terminology, unique up to an
isotopy, this important fact went largely unnoticed, 
although there were
exceptions, \hbox{cf.~\cite[Theorem~5]{pet}}.

The first aim of this note is to give, in Section~\ref{Tri dny je moje standardni prerioda.}, a simple and transparent proof
of this strong version of the minimal model theorem, based on the
approach of \cite{tr}. This unique-up-to-isotopy minimal model of a dg
algebra $(A,d,\cdot)$ with cohomology $H$
is given by a~connecting \Ainfty-morphism
\begin{equation}
\label{Prispel jsem na Kelimka.}
\psi_\infty = (\psi,\psi_2,\psi_3,\ldots) : (H,0,\mu_2,\mu_3,\ldots) 
\longrightarrow (A,\di,\cdot,0,0,\ldots)
\end{equation}
whose linear part $\psi: H \to A$ is
such that $\psi(h)$ is, for each $h\in H$, a
representative of the cohomology class $h$. For lack of a better
name, we will call $(H,0,\mu_2,\mu_3,\ldots)$  
the \cmm\ of~$A$.

In Section~\ref{Nastrikal jsem si dvirka na letadylko.} we pose the
question what can be said about the value $\mu_3(x,y,z)$ of the
trilinear product of a \cmm\ for concrete elements $x,y,z \in
H$. As the model is unique only up to an isotopy, the value
$\mu_3(x,y,z)$ is not defined unless the concrete \cmm\ whose part 
$\mu_3$ is has been specified. We will show that attempting to extract
a well-defined operation from $\mu_3(x,y,z)$ leads inevitably to the
recovery of the triple Massey product $\<x,y,z\>$ even if we have
never previously encountered it.

In Section~\ref{Zemrel Jiri Laufer co mi pripominal Jirku Chramostu.}
we restrict our attention to strictly defined Massey products.  Recall
that the $n$-ary product $\<\Rada x1n\>$ is strictly defined if
all lower Massey products of substrings of $\Rada x1n$ contain only
the value $0$.  We prove that, under a natural linear independence
assumption, the set $\<\Rada x1n\>$ is, modulo a specific sign, equal
to the set of values $\mu_n(\Rada x1n)$, with $\mu_n$'s running over
$n$-ary products of \cmm{s} of $A$, cf.~Theorem~\ref{Zitra jdu s
  Jarkou na CEZ.} and Corollary~\ref{Zaletam si v patek s
  letadylkem?}. As noted by P.~May 
in~\cite{Matrix}, citing: ``strictly defined matric Massey products
provide a more satisfactory generalization of triple products than do
arbitrary matric Massey products.''  We therefore consider the
restriction to strictly defined products in
Section~\ref{Zemrel Jiri Laufer co mi
  pripominal Jirku Chramostu.} to be fully justified.  
We do not think that the results of this part can be
meaningfully generalized to the general case.
\begin{center}
  -- -- -- -- -- 
\end{center}

The reason why any relationship between \cmm{s} and Massey products
exists at all is that the values of the components of the connecting
\Ainfty-morphism~(\ref{Prispel jsem na Kelimka.}) 
in cases of interest determine a defining system for a Massey
product in a very explicit way, 
given by the formulas of~Lemma~\ref{Krtecek mi pomaha.}. This lemma is
also the main technical result of this note.
 
So far we emphasized the (non)uniqueness of the canonical 
minimal  models, but have not
mentioned how to find them in practice. There is either the `classic'
inductive construction by T.~Kadeishvili~\cite{Kadei}, or an explicit 
closed formula given in~\cite{tr}, which depends on $\psi:H \to A$ 
as in~\eqref{Prispel jsem
  na Kelimka.}, on its right homotopy inverse $\pi$ and on a cochain
homotopy between the identity endomorphism $\id_A$ of $A$ 
and $\psi\pi$. Our existence proof
 in Section~\ref{Tri dny je moje standardni prerioda.} uses the latter 
approach. 

\vskip .5em
\noindent 
{\bf Signs.}
For the defining system of Massey
products, we use the standard sign convention 
and indeed the only one we know
of. The  trick of employing the sign reversal operator $a \mapsto
\znamenko a a$ is ingenious and nothing better could be invented. 

There are two natural sign conventions for \Ainfty-algebras 
and their morphisms, both based on a specific
translation between the \Ainfty-language and language of differential
graded tensor coalgebras~\cite[Example~II.3.90]{MSS}. These 
conventions are related by
overall signs, which typically look like $(-1)^{\frac{n(n-1)}2}$, where
$n$ is the arity of the corresponding operation. The dichotomy arises
from the fact that the $n$th tensor product \hbox{$\uparrow^{\ot n}$} of the
suspension operator is \underline{not} the inverse of the  $n$th tensor product 
\hbox{$\downarrow^{\ot n}$} of the
desuspension, but
\[
\uparrow^{\ot n}
\circ
\downarrow ^{\ot n}= \downarrow^{\ot n} \circ \uparrow^{\ot n}
= (-1)^{\frac{n(n-1)}2} \cdot \id.
\]
We choose the convention of~\cite{tr}.

Not surprisingly, the sign convention for Massey products is
not compatible with either of the two conventions for
\Ainfty-algebras, since they stem from different considerations.
Therefore the transmutation of the components of the connecting
\Ainfty-morphism~(\ref{Prispel jsem na Kelimka.}) into a defining system
given in Lemma~\ref{Krtecek mi pomaha.}
includes nontrivial sign factors. 
On the other hand, the choice of signs for
\Ainfty-objects is quite flexible, as explained in the first paragraph of page~141
in~\cite{tr}. We therefore believe that one can find 
a convention
so that there are no signs in the conversion
of~Lemma~\ref{Krtecek mi pomaha.}. However, we do not think that
there will be any use for such a convention except in the context of
this note.

\vskip .5em
\noindent 
{\bf Acknowledgment.} 
We are indebted to G.~Papayanov
who explained to us that the `strong minimal model theorem' was not our
invention, but was implicit in a 1980 paper~\cite{Kadei} by
T.~Kadeishvili.  
We would also like to thank Universidad de M\'alaga,
namely U.\ Buijs, J.M.\ Moreno-Fern\'andez and A.\ Murillo, for for
their gracious hospitality during the period when the first version of this note was
conceived.  Special appreciation is also due to J.M.\
Moreno-Fern\'andez for identifying several misprints in the MS
and to D.~Petersen who informed us that 
the first occurrence of the
  term "isotopy" in the context of this note  was probably 
the 2015 article~\cite{DSV}.

\vskip .5em
\noindent 
{\bf Conventions.} 
All algebraic objects in this note will be defined over a field
$\bbk$. By $\Lin^s(U,V)$ we denote the space of 
$\bbk$-linear maps $U \to V$ raising the degree by $s$.
The degree of a homogeneous element $x$ of a graded space $X$
will be denoted by $|x|$, although in expressions such as~$(-1)^{|x|}$ we
omit the vertical bars and write  $(-1)^{x}$ instead.
Writing $x \in X$ we assume that $x$ is homogeneous, and $\overline
x$ will denote $(-1)^x \cdot x$.

For \Ainfty-algebras and their morphisms we rely on the
sign conventions given
in~\cite[page 141]{tr}.
However, we will work in the cohomological scenario, so 
that all differentials are of 
degree~$+1$, as is traditional in the context of Massey products.  
The degrees of other objects should to be reversed accordingly. 
Cochain complexes, possibly unbounded, will be typically denoted by 
$(A,\di)$, $(B,\di)$,~\&c. We will use the
same symbol $\di$ for all  differentials, since they will always be
determined by their underlying spaces. 
We will denote by $[x]$ the equivalence class of $x$ in a quotient space.
In particular, $[a]$ will denote the
cohomology class of a cocycle $a$.
The abbreviation {\em dg algebra\/} stands for a differential graded
associative $\bbk$-algebra.  Dg algebras will be treated as
\Ainfty-algebras with all ternary and higher products trivial.

An \hbox{$A_\infty$-algebra} $A =
(A,\di,\mu_2,\mu_3,\ldots)$  will be often written as $A = (A,\di,\bfmu)$,
with $\bfmu$ serving as the collective symbol for the  higher products $\mu_k:
A^{\ot k } \to A$, $k \geq 2$.
An \Ainfty-morphism 
\
\[
\vartheta_\infty = (\vartheta,\vartheta_2,\vartheta_3,\ldots) : (A,\di,\mu_2,\mu_3,\ldots) \longrightarrow
(B,\di,\nu_2,\nu_3,\ldots),\ \vartheta : A \to B,\
\vartheta_k:A^{\ot k} \to B,\ k \geq 2,
\]
will be abbreviated accordingly by $\vartheta_\infty :  (A,\di,\bfmu) \to
(B,\di,\bfnu)$. Such an \Ainfty-morphism  is
an {\em isotopy \/} if its linear part
$\vartheta$ equals  the identity automorphism $\id_A$ of
$A$. 
Then of course $(B,\di) = (A,\di)$. 
We will call the \Ainfty-morphism $\vartheta_\infty$ above
an {\/\em extension\/} of the chain map $\vartheta: (A,\di) \to (B,\di)$.

\section{Strong minimal model theorem}
\label{Tri dny je moje standardni prerioda.}

\lettrine{\color{red} T}{his}
section brings back to life, in the form of the `strong
minimal model theorem,' a~remark dropped in the
last paragraph of page~236 in the classic article~\cite{Kadei} of 
Tornike~Kadeishvili. However, our proof of his result will be different
than the inductive Tornike's construction.

\begin{definition}
\label{Zatahne se jim to?}
Let $A = (A,d,\cdot)$ be a dg algebra  and
$H = H(A,\di)$ its cohomology. We call an \hbox{\Ainfty-structure}
$(H,0,\mu_2,\mu_3,\ldots)$ on $H$ the {\/\em canonical minimal
model\/ } of $A$ if there exists an associated {\em connecting \Ainfty-morphism}
\begin{equation}
\label{Vaham jestli mam jet ale asi nepojedu.}
\psi_\infty = (\psi,\psi_2,\psi_3,\ldots) : (H,0,\mu_2,\mu_3,\ldots) 
\longrightarrow (A,\di,\cdot,0,0,\ldots) 
\end{equation}
such that $\psi: H \to A$ sends each $x \in H(A,\di)$ to a 
representative of its
cohomology class, that is 
\begin{equation}
\label{Pojedeme s Jarkou do OBI.}
[\psi(h)] = h, \hbox { for each } \ h \in H(A,\di).
\end{equation}
We will call an operation $\mu_n : H^{\ot n} \to
H$ is a {\em \isot\ $n$-ary product \/} if it is the $n$th operation of some \cmm\  of
$\dga$.
\end{definition}
 
It follows from the definition of \Ainfty-morphisms and
property~(\ref{Pojedeme s Jarkou do OBI.}) of the map
$\psi : H \to A$ that
an operation \hbox{$\mu_2 : H^{\ot 2} \to
H$} is a canonical binary product if and only if it is the standard induced
multiplication on the cohomology. 
Canonicity is important here,
without this assumption, there may exist  $\mu_2$'s in 
(non-canonical) minimal models which differ from the 
induced~multiplication, as illustrated by  equation~\eqref{Dal jsem si
  do kavy vice slehacky nez obvykle.} in
Example~\ref{Mam predplatne Vesmiru.}.

\begin{theorem}[Strong minimal model theorem]
\label{Pujdeme s Jarkou k Mechacum.} 
Every dg associative algebra admits a canonical minimal model which is
unique up to isotopy.
\end{theorem}

The principal difference against the standard minimal model theorem as
formulated for instance in~\cite{Keller} and repeated at several places
since (cf.~e.g.\ ~\cite[Theorem~1.1]{BMM}), is the word ``isotopy'' instead of
 ``isomorphism.'' 

\begin{proof}[Proof of Theorem~\ref{Pujdeme s Jarkou k Mechacum.}]
Let us prove the existence first. 
As we work over a field, there clearly exists a cochain 
map \hbox{$\pp : (A,\di) \to (H,0)$} which, restricted to  ${\rm Ker}\
\di \subset A$, is the canonical
projection to~$H$,  together with  its left
homotopy inverse $\ps : (H,0) \to (A,\di)$ such that \hbox{$\pp\ps = 
\id_H$}. 
This brings us to the special situation of the celebrated 
Theorem~5 of~\cite{tr}, i.e.\ there exists a diagram
\[
\xymatrix@C=5em{
 *{\quad \quad  (A,\di) \ } \ar@(ul,dl)[]_{h}}
\hskip -3.8em
\xymatrix@C=5em{
\rule{3em}{0em}\rule{0em}{2em}  \ar@<0.2em>@/^.8em/[r]^\pi
& {(H,0)\, ,} \ar@<0.2ex>@/^.8em/[l]_\psi
}\
\psi\pi - \id_A = \di h + h \di.
\]
By Theorem~5  loc.cit., the chain map $\psi$ extends to an
\Ainfty-morphism 
$\psi_\infty:  (H,0,\bfmu) \to \C$, which is the required \cmm.

Let $\psi'_\infty  : (H,0,\bfmu')  \to (A,\di,\cdot)$, $\psi''_\infty  :
(H,0,\bfmu'')  \to (A,\di,\cdot)$ 
be two \cmm{s} of~$A$. By assumption, for each $x\in H$, the
difference $\psi''(x) -
\psi'(x)$ is a coboundary, 
so $\psi''$ and $\psi'$ are homotopic since we are working over a
field. We find ourselves in the situation of 
Item (ii) of~\hbox{\cite[Proposition~6]{bifib}} 
expressed by the diagram
\begin{equation}
\label{Snad desetkrat!}
\xymatrix@C=3em{
(H,0,\bfmu')   \ar[r]^{\psi'_\infty}     \ar@{-->}[d]_S & (A,\di,\cdot) \ar@{=}[d]
\\
(H,0,\bfmu'') \ar[r]^{\psi''_\infty}    & \,(A,\di,\cdot).
}
\end{equation}
Item (i) of this proposition produces the required isotopy $S :
(H,0,\bfmu') \to  (H,0,\bfmu'')$.
\end{proof}

\begin{remark}
It can be proved, using~\cite[Proposition~6]{bifib} again, that the
isotopy $S$ in~(\ref{Snad desetkrat!}) can be chosen so that the
diagram commutes up to an \Ainfty-homotopy.
\end{remark}

\begin{example}
\label{Mam predplatne Vesmiru.}
Let $A = (A,\cdot,0)$ be a dg algebra concentrated in degree $0$,
i.e.\ a traditional associative $\bbk$-algebra. Since 
the differential is absent, 
its cohomology $H(A,0)$ is canonically isomorphic
to~$A$. A general minimal model of $(A,\cdot,0)$  is of the form
\[
\psi: (A,*,0) \longrightarrow (A,\cdot,0),
\]
where $\psi: A \to A$ is an arbitrary automorphism and $*: A \ot A \to A$ the
associative multiplication
\begin{equation}
\label{Dal jsem si do kavy vice slehacky nez obvykle.}
x * y := \psi^{-1} \left(\psi(x) \cdot \psi(y)\right), \  x,y \in A.
\end{equation}

By definition, this minimal model is canonical if $\psi = \id_A$, in
which case $*$ equals the original multiplication of $A$. The
uniqueness part is trivial, since the only isotopy is the
identity automorphism. This simple example illustrates the
difference between the standard minimal model theorem and the strong
one -- while general minimal models of $(A,\cdot ,0)$ 
are parametrized by automorphisms  $\psi: A
\to A$, there is exactly one \cmm\  of $(A,\cdot ,0)$.
\end{example}

\section{Triple Massey products rediscovered}
\label{Nastrikal jsem si dvirka na letadylko.}

\lettrine{\color{red} S}{uppose}
that  want to
study the homotopy type of a dg algebra $(A,\di)$ by means of some
operations on its cohomology $H = (A,\di)$, but have never heard 
of Massey products. The aim of this section is to show how an
invisible hand leads 
us to rediscover them.  Since we are not
completely ignorant  we know
that $H$ carries an \Ainfty-structure $(H,0,\bfmu)$ which captures the homotopy
type of $(A,\di)$, so this {\/\em \Ainfty-model \/} is the obvious candidate. 
Thanks to the strong minimal model theorem, it
is unique up to an isotopy, so it has to be considered as belonging to
a, typically infinite-dimensional, vector space modulo
an action of an affine group. Moduli spaces of this
type were indeed studied in the early years of rational homotopy
theory, cf.~\cite{Kocour Felix} 
for example, but we want something more tractable. 
A truncation of the
\cmm\ is the obvious~choice.

The key to the success of this approach is to understand the indeterminacy of the
individual \Ainfty-products.  
Suppose therefore that we are given two \Ainfty-models 
$(H,0,\bfmu')$ and $(H,0,\bfmu'')$ related by an isotopy
\[
\tau_\infty = (\tau,\tau_2,\tau_3,\ldots)   
: (H,0,\mu'_2,\mu'_3,\ldots) \longrightarrow (H,0,\mu''_2,\mu''_3,\ldots),
\]
where $\tau = \id_H$. The axioms of 
\Ainfty-morphism~\cite[Section~2]{tr}, in our situation where the differentials are
trivial, lead
in arities $2$ and $3$ to
\begin{subequations} 
\begin{align}
\label{V sobotu jsem si zaletal 3 a pul hodiny na letadelku.}
\mu''_2(\tau,\tau) - \tau \mu'_2 &=  0, \ \hbox { and}
\\
\label{Jsem zvedav co mi napsal.}
\tau \mu'_3 - \mu''_3 (\tau,\tau,\tau)
&= - \mu''_2(\tau,\tau_2) + \mu''_2(\tau_2,\tau)
+ \tau_2(\mu'_2 , \id_H) - \tau_2(\id_H, \mu'_2).
\end{align}
\end{subequations}
Equation~(\ref{Jsem zvedav co mi napsal.}) 
is illustrated in Figure~\ref{Beata neco popletla a ja za to mam
  pykat.}; we found images of this type very helpful for understanding
\Ainfty-algebras and their morphisms.
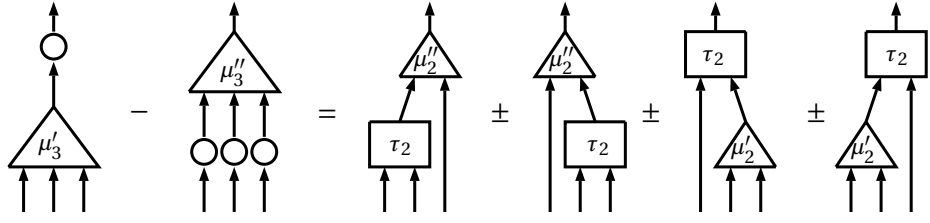
\begin{figure} 
\[
\psscalebox{1.0 1.0} 
{
\begin{pspicture}(0,-.8)(12.26,1.54792)
\psline[linecolor=black, linewidth=0.04](0.04000015,-0.55347914)(1.2400001,-0.55347914)(0.64000016,0.24652085)(0.04000015,-0.55347914)(0.04000015,-0.55347914)
\psline[linecolor=black, linewidth=0.04, arrowsize=0.05291667cm 2.0,arrowlength=1.4,arrowinset=0.0]{->}(0.64000016,0.24652085)(0.64000016,0.84652084)
\psline[linecolor=black, linewidth=0.04, arrowsize=0.05291667cm 2.0,arrowlength=1.4,arrowinset=0.0]{->}(0.24000016,-1.1534791)(0.24000016,-0.55347914)
\psline[linecolor=black, linewidth=0.04, arrowsize=0.05291667cm 2.0,arrowlength=1.4,arrowinset=0.0]{->}(0.64000016,-1.1534791)(0.64000016,-0.55347914)
\psline[linecolor=black, linewidth=0.04, arrowsize=0.05291667cm 2.0,arrowlength=1.4,arrowinset=0.0]{->}(1.0400002,-1.1534791)(1.0400002,-0.55347914)
\rput[bl](1.6400001,0.0652084){$-$}
\psline[linecolor=black, linewidth=0.04](2.44,0.44652084)(3.64,0.44652084)(3.0400002,1.2465209)(2.44,0.44652084)(2.44,0.44652084)
\psline[linecolor=black, linewidth=0.04, arrowsize=0.05291667cm 2.0,arrowlength=1.4,arrowinset=0.0]{->}(2.64,-0.15347916)(2.64,0.44652084)
\psline[linecolor=black, linewidth=0.04, arrowsize=0.05291667cm 2.0,arrowlength=1.4,arrowinset=0.0]{->}(3.0400002,-0.15347916)(3.0400002,0.44652084)
\psline[linecolor=black, linewidth=0.04, arrowsize=0.05291667cm 2.0,arrowlength=1.4,arrowinset=0.0]{->}(3.44,-0.15347916)(3.44,0.44652084)
\psline[linecolor=black, linewidth=0.04](4.84,0.046520844)(5.6400003,0.046520844)(5.6400003,-0.55347914)(4.84,-0.55347914)(4.84,0.046520844)(4.84,0.046520844)
\psline[linecolor=black, linewidth=0.04, arrowsize=0.05291667cm 2.0,arrowlength=1.4,arrowinset=0.0]{->}(5.04,-1.1534791)(5.04,-0.55347914)
\psline[linecolor=black, linewidth=0.04, arrowsize=0.05291667cm 2.0,arrowlength=1.4,arrowinset=0.0]{->}(5.44,-1.1534791)(5.44,-0.55347914)
\psline[linecolor=black, linewidth=0.04, arrowsize=0.05291667cm 2.0,arrowlength=1.4,arrowinset=0.0]{->}(5.84,-1.1534791)(5.84,0.64652085)
\psline[linecolor=black, linewidth=0.04, arrowsize=0.05291667cm 2.0,arrowlength=1.4,arrowinset=0.0]{->}(5.2400002,0.046520844)(5.44,0.64652085)
\psline[linecolor=black, linewidth=0.04](5.2400002,0.64652085)(6.04,0.64652085)(5.6400003,1.2465209)(5.2400002,0.64652085)
\psline[linecolor=black, linewidth=0.04, arrowsize=0.05291667cm 2.0,arrowlength=1.4,arrowinset=0.0]{->}(5.6400003,1.2465209)(5.6400003,1.6465209)
\psline[linecolor=black, linewidth=0.04](8.24,0.046520844)(7.44,0.046520844)(7.44,-0.55347914)(8.24,-0.55347914)(8.24,0.046520844)(8.24,0.046520844)
\psline[linecolor=black, linewidth=0.04, arrowsize=0.05291667cm 2.0,arrowlength=1.4,arrowinset=0.0]{->}(8.04,-1.1534791)(8.04,-0.55347914)
\psline[linecolor=black, linewidth=0.04, arrowsize=0.05291667cm 2.0,arrowlength=1.4,arrowinset=0.0]{->}(7.6400003,-1.1534791)(7.6400003,-0.55347914)
\psline[linecolor=black, linewidth=0.04, arrowsize=0.05291667cm 2.0,arrowlength=1.4,arrowinset=0.0]{->}(7.2400002,-1.1534791)(7.2400002,0.64652085)
\psline[linecolor=black, linewidth=0.04, arrowsize=0.05291667cm 2.0,arrowlength=1.4,arrowinset=0.0]{->}(7.84,0.046520844)(7.6400003,0.64652085)
\psline[linecolor=black, linewidth=0.04](7.84,0.64652085)(7.04,0.64652085)(7.44,1.2465209)(7.84,0.64652085)
\psline[linecolor=black, linewidth=0.04, arrowsize=0.05291667cm 2.0,arrowlength=1.4,arrowinset=0.0]{->}(7.44,1.2465209)(7.44,1.6465209)
\psline[linecolor=black, linewidth=0.04](9.84,1.2465209)(9.04,1.2465209)(9.04,0.64652085)(9.84,0.64652085)(9.84,1.2465209)(9.84,1.2465209)
\psline[linecolor=black, linewidth=0.04](10.24,-0.55347914)(9.440001,-0.55347914)(9.84,0.046520844)(10.24,-0.55347914)
\psline[linecolor=black, linewidth=0.04, arrowsize=0.05291667cm 2.0,arrowlength=1.4,arrowinset=0.0]{->}(9.84,0.046520844)(9.64,0.64652085)
\psline[linecolor=black, linewidth=0.04, arrowsize=0.05291667cm 2.0,arrowlength=1.4,arrowinset=0.0]{->}(9.64,-1.1534791)(9.64,-0.55347914)
\psline[linecolor=black, linewidth=0.04, arrowsize=0.05291667cm 2.0,arrowlength=1.4,arrowinset=0.0]{->}(10.04,-1.1534791)(10.04,-0.55347914)
\psline[linecolor=black, linewidth=0.04, arrowsize=0.05291667cm 2.0,arrowlength=1.4,arrowinset=0.0]{->}(9.24,-1.1534791)(9.24,0.64652085)
\psline[linecolor=black, linewidth=0.04, arrowsize=0.05291667cm 2.0,arrowlength=1.4,arrowinset=0.0]{->}(9.440001,1.2465209)(9.440001,1.6465209)
\psline[linecolor=black, linewidth=0.04](11.440001,1.2465209)(12.24,1.2465209)(12.24,0.64652085)(11.440001,0.64652085)(11.440001,1.2465209)(11.440001,1.2465209)
\psline[linecolor=black, linewidth=0.04](11.04,-0.55347914)(11.84,-0.55347914)(11.440001,0.046520844)(11.04,-0.55347914)
\psline[linecolor=black, linewidth=0.04, arrowsize=0.05291667cm 2.0,arrowlength=1.4,arrowinset=0.0]{->}(11.440001,0.046520844)(11.64,0.64652085)
\psline[linecolor=black, linewidth=0.04, arrowsize=0.05291667cm 2.0,arrowlength=1.4,arrowinset=0.0]{->}(11.64,-1.1534791)(11.64,-0.55347914)
\psline[linecolor=black, linewidth=0.04, arrowsize=0.05291667cm 2.0,arrowlength=1.4,arrowinset=0.0]{->}(11.24,-1.1534791)(11.24,-0.55347914)
\psline[linecolor=black, linewidth=0.04, arrowsize=0.05291667cm 2.0,arrowlength=1.4,arrowinset=0.0]{->}(12.04,-1.1534791)(12.04,0.64652085)
\psline[linecolor=black, linewidth=0.04, arrowsize=0.05291667cm 2.0,arrowlength=1.4,arrowinset=0.0]{->}(11.84,1.2465209)(11.84,1.6465209)
\pscircle[linecolor=black, linewidth=0.04, dimen=outer](2.64,-0.35347915){0.2}
\pscircle[linecolor=black, linewidth=0.04, dimen=outer](3.0400002,-0.35347915){0.2}
\pscircle[linecolor=black, linewidth=0.04, dimen=outer](0.64000016,1.0465208){0.2}
\pscircle[linecolor=black, linewidth=0.04, dimen=outer](3.44,-0.35347915){0.2}
\psline[linecolor=black, linewidth=0.04, arrowsize=0.05291667cm 2.0,arrowlength=1.4,arrowinset=0.0]{->}(3.44,-1.1534791)(3.44,-0.55347914)
\psline[linecolor=black, linewidth=0.04, arrowsize=0.05291667cm 2.0,arrowlength=1.4,arrowinset=0.0]{->}(3.0400002,-1.1534791)(3.0400002,-0.55347914)
\psline[linecolor=black, linewidth=0.04, arrowsize=0.05291667cm 2.0,arrowlength=1.4,arrowinset=0.0]{->}(2.64,-1.1534791)(2.64,-0.55347914)
\psline[linecolor=black, linewidth=0.04, arrowsize=0.05291667cm 2.0,arrowlength=1.4,arrowinset=0.0]{->}(3.0400002,1.2465209)(3.0400002,1.6465209)
\psline[linecolor=black, linewidth=0.04, arrowsize=0.05291667cm 2.0,arrowlength=1.4,arrowinset=0.0]{->}(0.64000016,1.2465209)(0.64000016,1.6465209)
\rput[bl](0.43,-0.45){\scriptsize$\mu_3'$}
\rput[bl](2.85,0.52084){\scriptsize$\mu_3''$}
\rput[bl](5.42,0.68){\scriptsize$\mu''_2$}
\rput[bl](5.10002,-0.4){\scriptsize$\tau_2$}
\rput[bl](7.22,0.68){\scriptsize$\mu''_2$}
\rput[bl](7.73,-0.4){\scriptsize$\tau_2$}
\rput[bl](9.26,.8){\scriptsize$\tau_2$}
\rput[bl](9.64,-0.5){\scriptsize$\mu'_2$}
\rput[bl](11.7,0.8){\scriptsize$\tau_2$}
\rput[bl](11.235,-0.5){\scriptsize$\mu'_2$}
\rput[bl](4.1400002,0.04652085){$=$}
\rput[bl](6.44,0.04652085){$\pm$}
\rput[bl](8.440001,0.046520844){$\pm$}
\rput[bl](10.64,0.046520844){$\pm$}
\end{pspicture}
}
\]
\caption{Flow diagram symbolizing equation~\eqref{Jsem zvedav co mi
    napsal.};   the white circles stand for~$\tau$.
\label{Beata neco popletla a ja za to mam pykat.}}  
\end{figure}

Equation~\eqref{V sobotu jsem si zaletal 3 a pul
  hodiny na letadelku.} with  $\tau =
\id_H$ shows that
$\mu'_2 = \mu''_2$, so $\mu_2$ is unique and 
is equal to the multiplication induced by the multiplication of the dg
algebra $A$. Thus~\eqref{Jsem zvedav co mi napsal.} with $\mu'_2 =
\mu''_2 =: \mu_2$ and   $\tau =
\id_H$~gives
\begin{equation}
\label{Pujdeme dnes k Pakousum?}
\mu'_3 - \mu''_3 
= - \mu_2(\id_H,\tau_2) + \mu_2(\tau_2,\id_H)
+ \tau_2(\mu_2 , \id_H) - \tau_2(\id_H, \mu_2).
\end{equation}
Equation~(\ref{Pujdeme dnes k Pakousum?}) 
implies that all canonical products
$\mu_3$'s determine a class in the quotient
\begin{equation}
\label{Budu si laminovat dvirka na podvozek.}
[\mu_3] \in \frac{\Lin^{-1}(H^{\ot 3},H)}{\Lin^{-1}(H^{\ot 2},H)},
\end{equation} 
where the elements $\tau_2$ of the affine group  $\Lin^{-1}(H^{\ot
  2},H)$ act on $\mu_3 \in \Lin^{-1}(H^{\ot 3},H)$ by the action
specified  by the right hand side of~(\ref{Pujdeme dnes k Pakousum?}).

We can do a little better, noticing that the right hand side
of~(\ref{Pujdeme dnes k Pakousum?}) describes 
the action of the Hochschild
differential  of the associative algebra $(H,\mu_2)$
on the $2$-cochain $\tau_2$. The \hbox{\Ainfty-axioms}
for $(H,0,\bfmu)$  imply
that all $\mu_3$'s must be Hochschild $3$-cocycles,  so~(\ref{Budu si
  laminovat dvirka na podvozek.}) can be improved~to
\[
[\mu_3] \in H_{\rm Hoch}^{3,-1}(H;H). 
\]
This invariant is called the {\/\em universal \/} Massey product,
cf.~the introduction of~\cite{Derived Muro} for the history of this
notion. {\em Nihil novi sub sole\/}, invariants of this
type have been studied in rational homotopy theory for a long time,
cf.~e.g.~\cite{HS}. 

Let us try to extract from the \Ainfty-model actual
`operations' acting on elements of $H$, at least in arities $2$ and
$3$. Since $\mu_2$ is unique, 
the value $\mu_2(x,y)$ is defined for arbitrary $x,y \in H$, 
and is equal to the induced
`cup product' $x \cdot y$. Let us inspect how $\mu_3(x,y,z)$ for given $x,y,z \in  H$ varies when $\mu_3$ moves through
the moduli space~(\ref{Budu si laminovat dvirka na podvozek.}).
To do this, we 
evaluate both sides of~\eqref{Pujdeme dnes k Pakousum?} at \hbox{$x \!\ot\! y\! \ot\!
z \in H^{\ot 3}$}, and denote for better readability 
the cup product $\mu_2$  by $\cdot$\ . The result is
\begin{equation}
\label{Pujdu si zabehat nebo se zblaznim.}
\mu''_3(x,y,z) - \mu'_3(x,y,z) =
-(-1)^x x \cdot \tau_2(y,z) + \tau_2(x,y)\cdot z + \tau_2(x\cdot y,z)  
-  \tau_2(x,y\cdot z)  .
\end{equation}

The indeterminacy of $\mu_3(x,y,z)$ 
is represented by the right hand side
of~\eqref{Pujdu si zabehat nebo se zblaznim.}. 
Since the values $\tau_2(x,y)$ and $\tau_2(y,z)$ may be practically
anything, the best we can say
about the first two terms is that they belong to the subspace 
\begin{equation}
\label{Jarka si nepretocila budika.}
x\cdot H^{y+z-1} + H^{x+y-1} \cdot z \subset
H^{x+y+z-1}.
\end{equation}
However, there is
nothing to be said a~priory about the remaining two terms in the
right hand side -- we have no control over $\tau_2(x\cdot
y,z)$ and $\tau_2(x, y\cdot z)$. The way how to save the day is to assume
that $x \cdot y =0$ and, likewise, that  $y \cdot z =0$.
Equation~\eqref{Pujdu si zabehat nebo se zblaznim.} then becomes
\begin{equation}
\label{Ulil jsem se z prednasky!}
\mu''_3(x,y,z) - \mu'_3(x,y,z) =
-(-1)^x x \cdot \tau_2(y,z) + \tau_2(x,y)\cdot z 
\end{equation}
illustrated in Figure~\ref{Beata neco popletla a ja za to mam pykat,
  sakra.}.
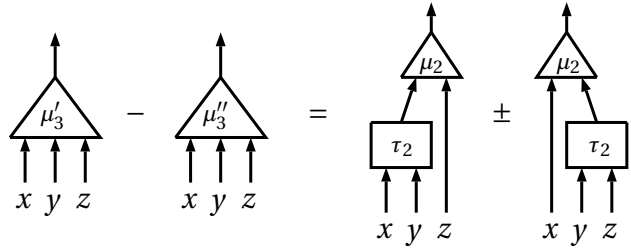
\begin{figure}
\[
\psscalebox{1.0 1.0} 
{
\begin{pspicture}(1.5,-1.15)(7.26,1.5534792)
\rput(0,.2){
\psline[linecolor=black, linewidth=0.04](0.04000015,-0.35347915)(1.2400001,-0.35347915)(0.64000016,0.44652084)(0.04000015,-0.35347915)(0.04000015,-0.35347915)
\psline[linecolor=black, linewidth=0.04, arrowsize=0.05291667cm 2.0,arrowlength=1.4,arrowinset=0.0]{->}(0.64000016,0.44652084)(0.64000016,1.0465208)
\psline[linecolor=black, linewidth=0.04, arrowsize=0.05291667cm 2.0,arrowlength=1.4,arrowinset=0.0]{->}(0.24000016,-0.9534792)(0.24000016,-0.35347915)
\psline[linecolor=black, linewidth=0.04, arrowsize=0.05291667cm 2.0,arrowlength=1.4,arrowinset=0.0]{->}(0.64000016,-0.9534792)(0.64000016,-0.35347915)
\psline[linecolor=black, linewidth=0.04, arrowsize=0.05291667cm 2.0,arrowlength=1.4,arrowinset=0.0]{->}(1.0400002,-0.9534792)(1.0400002,-0.35347915)
\psline[linecolor=black, linewidth=0.04](2.2400002,-0.35347915)(3.44,-0.35347915)(2.8400002,0.44652084)(2.2400002,-0.35347915)(2.2400002,-0.35347915)
\psline[linecolor=black, linewidth=0.04, arrowsize=0.05291667cm 2.0,arrowlength=1.4,arrowinset=0.0]{->}(2.44,-0.9534792)(2.44,-0.35347915)
\psline[linecolor=black, linewidth=0.04, arrowsize=0.05291667cm 2.0,arrowlength=1.4,arrowinset=0.0]{->}(2.8400002,-0.9534792)(2.8400002,-0.35347915)
\psline[linecolor=black, linewidth=0.04, arrowsize=0.05291667cm 2.0,arrowlength=1.4,arrowinset=0.0]{->}(3.2400002,-0.9534792)(3.2400002,-0.35347915)
\psline[linecolor=black, linewidth=0.04](4.84,-0.15347916)(5.6400003,-0.15347916)(5.6400003,-0.7534792)(4.84,-0.7534792)(4.84,-0.15347916)(4.84,-0.15347916)
\psline[linecolor=black, linewidth=0.04, arrowsize=0.05291667cm 2.0,arrowlength=1.4,arrowinset=0.0]{->}(5.04,-1.3534791)(5.04,-0.7534792)
\psline[linecolor=black, linewidth=0.04, arrowsize=0.05291667cm 2.0,arrowlength=1.4,arrowinset=0.0]{->}(5.44,-1.3534791)(5.44,-0.7534792)
\psline[linecolor=black, linewidth=0.04, arrowsize=0.05291667cm 2.0,arrowlength=1.4,arrowinset=0.0]{->}(5.84,-1.3534791)(5.84,0.44652084)
\psline[linecolor=black, linewidth=0.04, arrowsize=0.05291667cm 2.0,arrowlength=1.4,arrowinset=0.0]{->}(5.2400002,-0.15347916)(5.44,0.44652084)
\psline[linecolor=black, linewidth=0.04](5.2400002,0.44652084)(6.04,0.44652084)(5.6400003,1.0465208)(5.2400002,0.44652084)
\psline[linecolor=black, linewidth=0.04, arrowsize=0.05291667cm 2.0,arrowlength=1.4,arrowinset=0.0]{->}(5.6400003,1.0465208)(5.6400003,1.4465208)
\psline[linecolor=black, linewidth=0.04](8.24,-0.15347916)(7.44,-0.15347916)(7.44,-0.7534792)(8.24,-0.7534792)(8.24,-0.15347916)(8.24,-0.15347916)
\psline[linecolor=black, linewidth=0.04, arrowsize=0.05291667cm 2.0,arrowlength=1.4,arrowinset=0.0]{->}(8.04,-1.3534791)(8.04,-0.7534792)
\psline[linecolor=black, linewidth=0.04, arrowsize=0.05291667cm 2.0,arrowlength=1.4,arrowinset=0.0]{->}(7.6400003,-1.3534791)(7.6400003,-0.7534792)
\psline[linecolor=black, linewidth=0.04, arrowsize=0.05291667cm 2.0,arrowlength=1.4,arrowinset=0.0]{->}(7.2400002,-1.3534791)(7.2400002,0.44652084)
\psline[linecolor=black, linewidth=0.04, arrowsize=0.05291667cm 2.0,arrowlength=1.4,arrowinset=0.0]{->}(7.84,-0.15347916)(7.6400003,0.44652084)
\psline[linecolor=black, linewidth=0.04](7.84,0.44652084)(7.04,0.44652084)(7.44,1.0465208)(7.84,0.44652084)
\psline[linecolor=black, linewidth=0.04, arrowsize=0.05291667cm 2.0,arrowlength=1.4,arrowinset=0.0]{->}(7.44,1.0465208)(7.44,1.4465208)
\psline[linecolor=black, linewidth=0.04, arrowsize=0.05291667cm
2.0,arrowlength=1.4,arrowinset=0.0]{->}(2.8400002,0.44652084)(2.8400002,1.0465208)
}
\rput[bl](0.44,-.05){\scriptsize$\mu_3'$}
\rput[bl](2.6,-.05){\scriptsize$\mu_3''$}
\rput[bl](5.48,0.69){\scriptsize$\mu_2$}
\rput[bl](5.10002,-0.4){\scriptsize$\tau_2$}
\rput[bl](7.3,0.684){\scriptsize$\mu_2$}
\rput[bl](7.73,-0.4){\scriptsize$\tau_2$}
\rput(-11,0.4){
\rput[bl](11.1,-1.5){$x$}
\rput[bl](11.48,-1.6){$y$}
\rput[bl](11.9,-1.5){$z$}
}
\rput(-8.8,0.4){
\rput[bl](11.1,-1.5){$x$}
\rput[bl](11.48,-1.6){$y$}
\rput[bl](11.9,-1.5){$z$}
}
\rput(-6.2,0){
\rput[bl](11.1,-1.5){$x$}
\rput[bl](11.48,-1.6){$y$}
\rput[bl](11.9,-1.5){$z$}
}
\rput(-4,0){
\rput[bl](11.1,-1.5){$x$}
\rput[bl](11.48,-1.6){$y$}
\rput[bl](11.9,-1.5){$z$}
}
\rput[bl](1.57,0.04652085){$-$}
\rput[bl](4.0,0.04652085){$=$}
\rput[bl](6.44,0.04652085){$\pm$}
\end{pspicture}
}
\]
\caption{Flow diagram symbolizing equation~\eqref{Ulil jsem se z prednasky!}.   
\label{Beata neco popletla a ja za to mam pykat, sakra.}}  
\end{figure}
All canonical ternary products thus determine the equivalence class
\begin{equation}
\label{Vratila se zima.}
[\mu_3(x,y,z)] \in 
\frac {H^{x+y+z-1}}{\ \big(x\cdot H^{y+z-1} + H^{x+y-1} \cdot z\  \big
  )},\ \hbox { for $x,y,z \in H$ such that $x \cdot y=y \cdot z = 0$}.
\end{equation}
In the rest of this section we convince the reader that we have rediscovered
the triple Massey product.

The first indication is that both the triple Massey product $\<x,y,z\>$ and the
invariant in~\eqref{Vratila se zima.} are defined only for triples $x,y,z \in
H$ such that $x\cdot y = y \cdot z = 0$. If this condition is
satisfied, there exists a {\/\em
  defining system\/} 
\[
D = \{\aa uv \ | \ 1 \leq u \leq v \leq 3\}
\]
for $\<x,y,z\>$ which, by definition, consists of homogeneous elements of $A$ such
that $\aa 11,\aa 22,\aa 33$ are cocycles satisfying 
\begin{subequations}
\begin{equation}
\label{Udrzim se dnes?}
x = [\aa11],\ y = [\aa22],\ z = [\aa33],
\end{equation}
and $\aa12,\aa23$ are cochains such that 
\begin{equation}
\label{Dnes jsme nakupovali jako blaznivi.}
\di\aa12 = \oaa11 \aa23,\ \hbox { and }\  \di \aa23 = \oaa12 \aa33.
\end{equation}
\end{subequations}
The triple Massey product is then the subset $\<x,y,z\>  \subset H^{x+y+z-1}$ 
consisting of the cohomology classes of the cocycles 
\[
c(D) := \oaa 11 \aa 23  + \oaa 12\aa 33 \in A^{x+y+z-1},
\]
with $D$ running over all defining systems for $\<x,y,z\>$. Higher-order Massey
products are recalled in Section~\ref{Zemrel Jiri Laufer co mi
  pripominal Jirku Chramostu.} of this note.

It is well-known and could be easily checked that, 
given two defining systems $D'$ and $D''$ for $\<x,y,z\>$, the
cohomology class of the difference $c(D'')-c(D')$ belongs to  
the subset~(\ref{Jarka si nepretocila budika.}) which we have
already seen while discussing the indeterminacy of
$\mu_3(x,y,z)$.  
In other words,
\begin{equation}
\label{Tu dnesni lat jsem prosvihl, ale bude desna zima.}
b',b'' \in \<x,y,z\>  \ \Longrightarrow \   b''\!-b' \in x\cdot  H^{y+z-1} + H^{x+y-1} \cdot z.
\end{equation}
The right hand side of~(\ref{Tu dnesni lat jsem prosvihl, ale bude desna zima.}) imposes a constraint on
the form of the indeterminacy of $\<x,y,z\>$. 
In the ternary case, it actually {\/\em characterizes\/} the
indeterminacy of  $\<x,y,z\>$, as the following proposition shows.

\begin{proposition}
\label{Nejde mne internet, takze se nemusim trapit tim, jak je dnes krasne.}
Assume that the triple Massey product $\<x,y,z\>$ is defined, and $b
\in \<x,y,z\>$. Then $b + \xi \in \<x,y,z\>$ for each $\xi \in x\cdot
H^{y+z-1} + H^{x+y-1} \cdot z$. 
\end{proposition}

\begin{proof}
Let $D = \{\aa uv \ | \ 1 \leq u \leq v \leq 3\}$ be a defining
system~(\ref{Udrzim se dnes?})--\eqref{Dnes jsme nakupovali jako blaznivi.} 
such that $b$ is the cohomology class of $c(D)$, and
$\xi = x \cdot \xi_{2,3} +  \xi_{1,2} \cdot z$ for some  $ \xi_{1,2},  \xi_{2,3} \in
H$. Choose cocycles $s_{1,2}, s_{2,3} \in A$ such that 
$[s_{1,2}] = \overline\xi_{1,2}$ and $[s_{2,3}] = \xi_{2,3}$. Then
\[
D_\xi := \big\{ \
\aa11,\ \aa22,\ \aa33,\ \aa12 + s_{12},\ \aa23 + s_{23}\ \}
\]
is a defining system such that the cohomology class of $c(D_\xi)$ and
therefore also the associated Massey product 
is $b + \xi$.
\end{proof}

We conclude from~(\ref{Tu dnesni lat jsem prosvihl, ale bude desna
  zima.}) and Proposition~\ref{Nejde mne internet, takze se nemusim
  trapit tim, jak je dnes krasne.} 
that $\<x,y,z\>$ is an affine subspace of
$H^{x+y+z-1}$, which is a shift of the linear subspace $ x\cdot
H^{y+z-1} + H^{x+y-1} \cdot z$. 
The following elementary lemma characterizes 
affine subspaces of this form.

\begin{lemma}
\label{Bude termika ale strasna zima.}
Let $R \subset U$ be a given linear subspace of a vector space $U$. 
Then there is a natural one-to-one correspondence between
\begin{itemize}
\item[(i)]
points of the quotient $U/R$, and
\item [(ii)]
subsets $S \subset U$ with the property that (\,$b',b'' \in S
\Rightarrow b''\! -b' \in R$)\ \& \ (\,$b \in S, \xi \in R
\Rightarrow b+\xi \in S$).
\end{itemize}
This correspondence assigns to a subset $S$ in (ii) its image 
under the projection $U \twoheadrightarrow U/R$.
\end{lemma}

Applying Lemma~\ref{Bude termika ale strasna zima.} 
to $U = H^{x+y+z-1}$, $R = x\cdot H^{y+z-1} +
H^{x+y-1} \cdot z$ and $S = \<x,y,z\>$ we obtain

\begin{corollary}
In terms of the information provided, the triple Massey product
$\<x,y,z\>$
is equivalent to a single point, namely to the equivalence class
\begin{equation}
\label{Koupil jsem si tiramisu, je docela dobre.}
[\<x,y,z\>]
 \in 
\frac {H^{x+y+z-1}}{\ \big(x\cdot H^{y+z-1} + H^{x+y-1} \cdot z\  \big
  )},
\end{equation}
defined whenever $x \cdot y=y \cdot z = 0$.
\end{corollary}

It remains to prove that the equivalence class $[\<x,y,z\>]$ equals
$[\mu_3(x,y,z)]$ up to a specific sign.
To do so, consider a \cmm\ for $(A,\di,\cdot)$ as in~\eqref{Vaham jestli mam jet ale asi
  nepojedu.}. It will be established later
in the proof of Lemma~\ref{Krtecek mi pomaha.} 
in Section~\ref{Zemrel Jiri Laufer co mi
  pripominal Jirku Chramostu.}, that 
the associated connecting 
\Ainfty-morphism \hbox{$\psi : (H,0,\bfmu) \to (A,\di,\cdot)$} 
determines a defining system
$D_\psi$ for $\<x,y,z\>$ by
\begin{equation}
\label{Pak si tam jeste nalepim delfinka.}
\aa11 := \psi(x),\ \aa22 := \psi(y),\ \aa33 := \psi(z),\ 
\aa12 = -(-1)^x \psi_2(x,y),\ \aa 23 :=  -(-1)^y \psi_2(y,z)
\end{equation}
and that, moreover,
\begin{equation}
\label{Uz skaci prvni kumuly a ja jsem v Praze!}
c(D_\psi) =  (-1)^y \cdot  \mu_3(x,y,z).
\end{equation}

\begin{theorem}
\label{Strikam si dvirka podvozku.}
The equivalence classes~(\ref{Vratila se zima.}) and~(\ref{Koupil jsem si tiramisu, je docela dobre.}) are related by
$
[\<x,y,z\>] =   (-1)^y \cdot [\mu_3(x,y,z)].
$
\end{theorem}

\begin{proof}
An immediate consequence of~(\ref{Uz skaci prvni kumuly a ja jsem v Praze!}).  
\end{proof}

The conceptual explanation why Theorem~\ref{Strikam si dvirka
  podvozku.} holds is a close relation between the defining systems
for triple Massey products and the second Taylor coefficient of 
the connecting \Ainfty-morphism~\eqref{Vaham jestli mam jet ale asi nepojedu.}, 
spelled out in~\eqref{Pak
  si tam jeste nalepim delfinka.}. We will see in the next
section that a similar relation holds for higher-order Massey products
as well.   
\begin{center}
  -- -- -- -- -- 
\end{center}

There are estimates for the indeterminacy  of
the higher Massey 
product, cf.~\cite[Proposition~2.3]{Matrix}, which for strictly
defined ones acquires a relatively simple
form~\cite[Proposition~2.4]{Matrix}. This raises a provocative
question under what assumptions these estimates 
are sharp, that is, when an analog of
Proposition~\ref{Nejde mne internet, takze se nemusim trapit tim, jak
  je dnes krasne.} holds.

\section{Higher Massey products}
\label{Zemrel Jiri Laufer co mi pripominal Jirku Chramostu.}

\lettrine{\color{red} I}{n}
the first part of this section we recall the definition of
higher-order Massey products, as well as some properties of the
connecting \Ainfty-morphism~\eqref{Vaham jestli mam jet ale asi
  nepojedu.} that follow from the axioms. 
The second part contains the main technical result of this note.
As before, $A = (A,\di,\cdot)$ will be a fixed
dg associative algebra with the cohomology algebra $H =
(H,0,\cdot\,)$.

\vskip .5em
\noindent 
{\bf Recollections.}
We start by recalling the standard definition of Massey products,
following~\cite[page~6]{BMM} closely. 
Let $\Rada x1n \in H$ be cohomology classes. We call the scheme  
\[
D = \big\{\aa uv \in A \ \big| \ |\aa uv| = 
\textstyle\sum_{r=u}^v (|x_r|-1) +1,\
1 \leq u \leq v \leq n\big\}
\] 
a {\em defining system\/} for the Massey product $\<\Rada
x1n\>$ if  $\aa uu$, $u = \rada 1n$, is a cocycle representing~$x_u$
and, for each $1\leq u \leq v \leq n$, $v-u \leq n-2$,
\begin{subequations}
\begin{equation}
\label{Vcera jsem se vratil z Mercina.}
\di \aa uv = \sum_{r=u}^{v-1} \oaa ur \cdot \aa {r+1}v.
\end{equation}
Each defining system determines the cocycle
\begin{equation}
\label{prvni letosni kilometry na kole}
c(D) := \sum_{r=1}^{n-1} \oaa ur \cdot \aa {r+1}v \in 
A^{x_1 + \cdots + x_n -n+2}.
\end{equation}
\end{subequations}
The subset of cohomology classes
\[
\< \Rada x1n\> : =\big\{\, [c(D)] \ \big|  \hbox {
  $D$ is a defining system for $\Rada a1n$ }\big\} \subset H^{x_1 + \cdots + x_n -n+2}
\]
is called the $n$-fold {\em Massey product\/} of $\Rada x1n$.  
We put  $\< \Rada x1n\> : =
\emptyset$ if there is no defining system for $\Rada x1n$.

\begin{example}
The defining system for the triple Massey product $\<x,y,z\>$ 
was explicitly written
out in~\eqref{Udrzim se dnes?}--\eqref{Dnes jsme nakupovali jako blaznivi.}.
The defining system for the $4$-fold
Massey product $\<x_1,x_2,x_3,x_4\>$ is the scheme
\[
\begin{gathered}
x_1 = [\aa11],\ x_2 = [\aa22],\ x_3 = [\aa33],\ x_4 = [\aa44]
\\
\di \aa12 = \oaa11\aa22,\ \di \aa23 = \oaa22\aa33,\ \di \aa34 =
    \oaa33\aa44,
\\
\di \aa13 = \oaa11 \aa23 + \oaa12\aa33,\ \di \aa24 = \oaa22 \aa34 +
    \oaa23\aa44.
\end{gathered}
\]
The corresponding element of the Massey product $\<x_1,x_2,x_3,x_4\>$  
is the cohomology class of
\[
c(D)  = \oaa11\aa24 + \oaa12\aa34 + \oaa13\aa44
\in A^{x_1+x_2+x_3+x_4 -2}\ .
\]
\end{example}

\begin{remark}
An amazing trick is to organize a defining system $D$ for $\<\Rada
x1n\>$ 
to the matrix
\[
\bfD:= 
\begin{pmatrix}
0 & a_{1,1}  & a_{1,2} & \cdots  & a_{1,n-1} &0
\\
0 & 0  & a_{2,2} & \cdots  & a_{2,n-1} &a_{2,n}
\\
\vdots & \vdots & \vdots & \cdots  & \vdots & \vdots
\\
0 & 0  & 0 & \cdots  & a_{n-1,n-1} &a_{n-1,n}
\\
0 & 0  & 0 & \cdots  &0  &a_{n,n}
\\
0 & 0  & 0 & \cdots  &0  &0
\end{pmatrix}.
\]
Equations~(\ref{Vcera jsem se vratil z Mercina.})--(\ref{prvni
  letosni kilometry na kole}) then assume a concise form of a single equation
\[
c(\bfD) = - d\,  \bfD + \overline{\bfD} \cdot \bfD
\]
with $c(\bfD)$ the \hbox{$(n\!+\!1)\! \times\! (n\!+\!1)$} matrix whose only nontrivial entry
is $c(D)$ at the upper right corner, cf.~\cite[page~538]{Matrix}. 
In this disguise $c(\bfD)$
appears as the noncommutative curvature of the connection $\bfD$.
\end{remark}

We also need to recall, following~\cite[Section~2]{tr}, the properties
that the components of a connecting
\Ainfty-morphism $\psi_\infty : (H,0,\bfmu) \to (A,d,\cdot)$ 
in~\eqref{Vaham jestli mam jet ale asi nepojedu.}
should satisfy. The linear part $\psi: H \to A$ must be a chain map,
which in this case means that $\di\psi(h)
= 0$ for $h \in H$. For  $h_1,h_2 \in H$  we require 
\begin{equation}
\label{Petr byl podobny zoufalec jako ja.}
 \di \psi_2(h_1,h_2) =  - \psi(h_1)\cdot\psi(h_2)  +
                       \psi\mu_2(h_1,h_2),
\end{equation}
and for  $h_1,h_2,h_3 \in H$
\begin{align}
\label{Prohlidka uz ve ctvrtek.}
\di \psi_3(h_1,h_2,h_3) =&
-\psi_2(h_1,h_2)\cdot\psi(h_3) + 
 \znamenko {h_1} \psi(h_1)\cdot\psi_2(h_2,h_3)
\\
\nonumber 
&-\psi_2\big(\mu_2(h_1,h_2),h_3\big) + \psi_2\big(h_1,\mu_2(h_2,h_3)\big)  
+ \psi\mu_3(h_1,h_2,h_3).
\end{align}
For $\Rada h1n \in H$, $n \geq 4$, we demand
\begin{align}
\label{Zapomel jsem si kleste.}
 \di\psi_n(&\Rada h1n) = 
 - \znamenko{n(1+h_1)}  \ps(h_1) \cdot \psi_{n-1}(\Rada h2n)
-  \psi_{n-1}(\Rada h1{n-1}) \cdot \ps(h_n)
\\
\nonumber
&-  {\sum}_A
\znamenko{(j+1)(i+h_1+\cdots + h_i)}\psi_i(\Rada h1i)\cdot \psi_j(\Rada h{i+1}n)
\\
\nonumber
& - {\sum}_B 
\znamenko{i +n+ k(h_1 + \cdots + h_{i-1} + i)}
\psi_l\big(\Rada h1{i-1},\mu_k(\Rada hi{i+k-1}),\Rada h{i+k}n\big) 
\\ \nonumber 
&+\psi \mu_n(\Rada h1n),
\end{align}
where $A := \{i,j \geq 2 \ | \ i+j = n \}$ and $B := 
\{ l,k  \geq 2,\ i \geq 1 \ | \ l+k = n +1, \ i+k\leq n \}$.

\vskip .5em
\noindent 
{\bf The results.}
Let us fix, throughout the rest of this section,
an ordered list of cohomology classes 
$\Rada x1n \in H(A,\di)$, $n \geq 1$. 
We will very often need the sign
\begin{equation}
\label{pisecna boure}
\varepsilon = \varepsilon(\Rada x1n) : = (-1)^{
\frac{(n+2)(n+1)}2 + \sum_{i=1}^n (n+i)x_i
} \in \{-1,+1\} .
\end{equation}
It differs from the sign in~\cite[Theorem~A]{BMM} by the overall
factor $(-1)^{\frac{(n+2)(n+1)}2}$   due to the
different convention for \Ainfty-algebras used here. 

By a {\em subinterval\/} of $\Rada x1n$ of
length $k \geq 2$ we mean a sequence $\Rada xuv$
of $k$ consecutive entries of $\Rada x1n$. The canonical products
mentioned below are the operations introduced in Definition~\ref{Zatahne
  se jim to?}. Having said that, consider the following three situations.
\begin{itemize}
\item [(i)]
Massey products satisfy
$\<x_u,\ldots,x_v\> = \{0\}$  
for each subinterval $x_u,\ldots,x_v$  of $\Rada x1n$ of length $\leq n-1$,
\item[(ii)]
every \isot\  product $\mu_k$ vanishes at 
each subinterval $x_u,\ldots,x_v$ of $\Rada x1n$  of length $k \leq n-1$ and, finally,
\item[(iii)]
$
\big\{\ \varepsilon\cdot \mu_n(\Rada x1n) \ \big| \ \mu_n \hbox { \isot}\ \big\} \subset \<\Rada
x1n\>$,
where $\varepsilon$ is the sign in~(\ref{pisecna boure}).
\end{itemize}

The necessary condition for the
Massey product  $\<\Rada
x1n\>$ to exist is that 
$0 \in \<x_u,\ldots,x_v\>$ for each subinterval $\Rada xuv$ of $\Rada x1n$
of length $\leq n-1$. However, this condition need not be
sufficient, cf.~\cite{4fold} for a careful analysis of the
$n=4$ case.  On the other hand, the existence of  $\<\Rada
x1n\>$ is implied by the conditions formulated in (i).
 In this case we say,
following May~\cite{Matrix}, 
that the $n$-fold Massey product $ \<x_1,\ldots,x_n\>$ is {\em
  strictly} defined. The following theorem and Corollary~\ref{Zaletam si v patek s letadylkem?} are the
main technical achievements of this note.

\begin{theorem}
\label{Zitra jdu s Jarkou na CEZ.}
$\hbox{(i)} \Longrightarrow \hbox{(ii)} \Longrightarrow \hbox{(iii)}$.
\end{theorem}

The proof is postponed towards the end of this section. Let us formulate
also the following

\begin{proposition}[after~\cite{BMM}]
\label{Funguje uz laborator?}
Assume that, for each $1 \leq k \leq n-1$, the tensor products
\[
x_u \ot \cdots \ot x_v  \in H^{\otimes k}(A,d),\ v-u+1 = k,\
1 \leq u \leq v \leq n,
\]
are linearly independent.
Then, given  $x \in \<\Rada x1n\>$, there exists a \isot\ product
$\mu_n$ such that $x =\varepsilon \cdot \mu_n(\Rada x1n)$. 
Therefore
\[
\big\{\ \varepsilon \cdot \mu_n(\Rada x1n) \ \big| \ \mu_n \hbox { is \isot}\
\big\} \supset \<\Rada x1n\>.
\]
\end{proposition}

\begin{proof}
It is straightforward to verify that, under the linear independence
assumption, the inductive construction described in 
the proof of~\cite[Theorem~2.1(i)]{BMM} works, 
and the emerging \Ainfty-structure  is \isot, cf.~also the
corrigendum~\cite{errata}.
\end{proof}

\begin{corollary}
\label{Zaletam si v patek s letadylkem?}
Assume that the $n$-fold Massey product
$\<x_1,\ldots,x_n\>$  
is strictly defined and the assumption of Proposition~\ref{Funguje uz
  laborator?} is fulfilled. Then
\[
\big\{\ \varepsilon \cdot  \mu_n(\Rada x1n) \ \big| \ \mu_n \hbox { is \isot }\ \big\} = \<\Rada
x1n\>.
\]
\end{corollary}

\begin{proof}
An obvious combination of Theorem~\ref{Zitra jdu s Jarkou na CEZ.} and
Proposition~\ref{Funguje uz laborator?}.
\end{proof}

In words, under the assumptions of the corollary,
the set  $\<\Rada
x1n\>$ is exhausted by all possible values of $\varepsilon\cdot\mu_n(\Rada
x1n) = \mu_n(\rada {\varepsilon\cdot x_1}{x_n})$ with
\isot\ $\mu_n$. 
Put another way, the information given by $\<\Rada
x1n\>$ is the same as the information provided by the values
$\mu_n(\rada {\varepsilon\cdot x_1}{x_n})$ with $\mu_n$ running over all 
\isot\ products. The following result is the core of our proof of 
Theorem~\ref{Zitra jdu s Jarkou na CEZ.}.

\begin{lemma}[the bootstrap] 
\label{Krtecek mi pomaha.}
Let $(H,0,\bfmu) = (H,0,\mu_2,\mu_3,\ldots)$ be a canonical
minimal model
as in~(\ref{Zatahne se jim to?}). Suppose that each operation $\mu_k$ of
this structure with $k
\leq n-1$ satisfies
\begin{equation}
\label{Tech veci co uz jsem postracel!}
\mu_k( x_u,\ldots,x_v)= 0
\end{equation} 
for every subinterval $x_u,\ldots,x_v$ of $\Rada x1n$ of length $k$. Then
\begin{align}
\nonumber 
\aa uu & := \psi(x_u),
\\
\nonumber 
\aa u{u+1}& := -\znamenko{x_u} \psi_2(x_u,x_{u+1}),
\\
\nonumber 
\aa u{u+2}& := - \znamenko{x_{u+1}} \psi_3(x_u,x_{u+1},x_{u+2}),
\\
\nonumber 
\aa u{u+3}& := \znamenko{x_u+x_{u+2}} \psi_4(x_u,x_{u+1},x_{u+2},x_{u+3}),
\\
\nonumber 
& \hskip .85em  \vdots
\\
\label{Pujdu vytahnout baterku z nabijecky.}
\aa uv & := \znamenko{
\frac {(v-u)(v-u+1)}2 + \sum_{i=0}^{v-u} (u+v+i)x_{u+i}
}
\psi_{v-u+1}(\Rada xuv),\ v-u \geq 1,
\end{align}
is a defining system for the $n$-fold Massey product $\<\Rada x1n\>$, and 
\begin{equation}
\label{A Laurinka mi pomaha take.}
\varepsilon\cdot
\mu_n(\Rada x1n) \in \<\Rada x1n\>.
\end{equation}
\end{lemma}

\begin{proof}
We proceed by induction. The assumptions for $n=3$ are $\mu_2(x_1,x_2)
=  \mu_2(x_2,x_3) = 0$, so~(\ref{Petr byl podobny zoufalec jako ja.}) 
gives
\begin{equation}
\label{Ruldu jsem tri roky nevidel.}
\di \psi_2(x_1,x_2)  = -  \psi(x_1)\psi(x_2) \ \hbox { and } \
\di \psi_2(x_2,x_3)  =  - \psi(x_2)\psi(x_3).
\end{equation}
We claim that then
\begin{equation}
\label{Za chvili na CEZ}
\hbox{$\aa uu := \psi(x_u)$, $u = 1,2,3$, $a_{1,2} : =
-\znamenko{x_1}  \psi_2(x_1,x_2)$ \ and \  $a_{2,3} : = -\znamenko{x_2} 
  \psi_2(x_2,x_3)$}
\end{equation} 
is a defining system for $\<x_1,x_2,x_3\>$. 

Indeed, $[\aa uu] = [\psi(x_u)] = x_u$ by~\eqref{Pojedeme s Jarkou do
  OBI.}, so $\aa uu$ is a
representative of $x_u$, while~\eqref{Ruldu jsem tri roky nevidel.} is
translated to $\di \aa12 = \oaa 11 \aa 22$ and   $\di \aa23 = \oaa 22
\aa 33$ as required.
Under the vanishing   $\mu_2(x_1,x_2)
=  \mu_2(x_2,x_3) = 0$, equation~\eqref{Prohlidka uz ve ctvrtek.}
gives
\[
 \di  \psi_3(x_1,x_2,x_3)  = -  \psi_2(x_1,x_2)\psi(x_3) +
 \znamenko{x_1} \psi(x_1)\psi_2(x_2,x_3)
+ \psi\mu_3(x_1,x_2,x_3), 
\]
which is the same as
\[
\psi\mu_3(x_1,x_2,x_3) =  \di  \psi_3(x_1,x_2,x_3) 
- \znamenko{x_1}  \psi(x_1)\psi_2(x_2,x_3) + \psi_2(x_1,x_2)\psi(x_3).
\]
Both sides of the above equation are cocycles so, by~(\ref{Pojedeme s
  Jarkou do OBI.}) again,
\begin{align*}
[\psi\mu_3(x_1,x_2,x_3)] &= \mu_3(x_1,x_2,x_3) =
\big[- \znamenko{x_1}  \psi(x_1)\psi_2(x_2,x_3) +
                     \psi_2(x_1,x_2)\psi(x_3)\big]
\\
&= \big[\znamenko{x_1 + x_2} \aa11 \aa23 - \znamenko {x_1}\aa
     12\aa23\big]
=\znamenko{x_2}\big[\oaa11 \aa23 + \oaa12\aa33\big],
\end{align*}
therefore
\[
\znamenko{x_2} \mu_3(x_1,x_2,x_3) = 
 [\oaa11 \aa23 + \oaa12\aa33].
\]
Since $\varepsilon(x_1,x_2,x_3) = (-1)^{x_2}$, we verified that
$\varepsilon \cdot  \mu_3(x_1,x_2,x_3) \in \<x_1,x_2,x_3\>$ as claimed. 
This finishes the first
step of the induction.

Assume that we have proved the lemma for all arities $\leq n-1$ and
prove it for arity $n$. By assumption, all terms
containing $\mu_k$ with $k \leq n-1$  in~\eqref{Zapomel jsem si
  kleste.} vanish, so we are left with
\begin{align*}
 \di\psi_n(&\Rada x1n) = 
 - \znamenko{n(1+x_1)}  \ps(x_1) \cdot \psi_{n-1}(\Rada x2n)
-  \psi_{n-1}(\Rada x1{n-1}) \cdot \ps(x_n)
\\
\nonumber
&-  {\sum}_A
\znamenko{(j+1)(i+x_1+\cdots + x_i)}\psi_i(\Rada x1i)\cdot \psi_j(\Rada x{i+1}n)
\\ \nonumber 
&+\psi \mu_n(\Rada x1n),
\end{align*}
where $A$ is as in~(\ref{Zapomel jsem si
  kleste.}). Substituting~(\ref{Pujdu vytahnout baterku z nabijecky.})
to the above display, we obtain, after a lengthy but elementary
manipulation with signs,
\begin{equation}
\label{Budu si asi muset ten bowden objednat znovu.}
\psi\mu_n(\Rada x1n) =  \di \psi_n(\Rada x1n) + \varepsilon\cdot c(D),
\end{equation}
where $c(D)$ is the expression in~(\ref{prvni letosni kilometry na
  kole}) for the scheme
\begin{equation}
\label{Ja ten bowden se kterym jsem se tak dlouho delal asi ztratil!}
D = \big\{\aa uv \in A \ \big| \
1 \leq u \leq v \leq n\big\}
\end{equation}
with $\aa uv$ defined in~\eqref{Pujdu vytahnout baterku z nabijecky.}.
It remains to verify that~(\ref{Ja ten bowden se kterym jsem
  se tak dlouho delal asi ztratil!}) is a  defining system for the
Massey product $\<\Rada x1n\>$, that is, to show that it 
satisfies~\eqref{Vcera jsem se vratil z Mercina.}. 
Invoking the vanishing assumption~(\ref{Tech veci co uz
  jsem postracel!}) we conclude that~\eqref{Zapomel jsem si kleste.} 
gives, for each $k \leq n-1$,
\begin{align*}
 \di\psi_k(&\Rada xuv) = 
 - \znamenko{k(1+x_u)}  \ps(x_u) \cdot \psi_{k-1}(\rada{x_{u+1}}{u_v})
-  \psi_{k-1}(\rada {x_u}{x_{v-1}}) \cdot \ps(x_v)
\\
\nonumber
&-  {\sum}_A
\znamenko{(j+1)(i+x_u+\cdots + x_{u+i-1})}\psi_i(\Rada xu{u+i-1})
\cdot \psi_j(\Rada x{u+i}v),
\end{align*}
where  $A := \{i,j \geq 2 \ | \ i+j = k \}$.
A painstaking calculation shows that the substitution~\eqref{Pujdu
  vytahnout baterku z nabijecky.} converts the above display
to~\eqref{Vcera jsem se vratil z Mercina.}.  
The induction step is completed by passing to the cohomology classes at both
sides of~\eqref{Budu si asi muset ten bowden objednat znovu.},
which results in
$[\psi\mu_n(\Rada x1n)] = \mu_n(\Rada x1n) =   \varepsilon\cdot [c(D)]$.
\end{proof}

\begin{proof}[Proof of Theorem~\ref{Zitra jdu s Jarkou na CEZ.}] 
We will prove $\hbox{(i)} \Longrightarrow \hbox{(ii)}$
by induction. For $n=3$, which is the first nontrivial
  case, we need to show that the vanishing of $\<x_1,x_2\>$ and  
$\<x_2,x_3\>$, which are always defined,  
implies the vanishing of $\mu_2(x_1,x_2)$ and
$\mu_2(x_2,x_3)$. This follows from the equalities
\[
\mu_2(x_1,x_2) = \znamenko{x_1} \<x_1,x_2\> \ \hbox { and }
\ \mu_2(x_2,x_3) = \znamenko{x_2} \<x_2,x_3\>,
\]
which are the binary case of~\eqref{A Laurinka mi pomaha take.} and
can be verified directly.

Assume that we have proved $\hbox{(i)} \Longrightarrow \hbox{(ii)}$ 
for some $n \geq 3$ and
prove it for $n+1$. This amounts to proving the vanishing of
$\mu_n(\Rada x1n)$ and $\mu_n(\Rada x2{n+1})$, the vanishing on
shorter subintervals of $\Rada x1{n+1}$ follows from the induction
assumption.
Equation~\eqref{A Laurinka mi pomaha take.} of Lemma~\ref{Krtecek mi
  pomaha.}, combined with the hypothesis that $\<\Rada x1n\> = \{0\}$,
implies that $\mu_n(\Rada x1n)=0$. The equality  $\mu_n(\Rada
x2{n+1}) =0$ can be established similarly. The implication  
$\hbox{(ii)} \Longrightarrow \hbox{(iii)}$ follows from  
Lemma~\ref{Krtecek mi pomaha.} in a straightforward way. 
\end{proof}


\end{document}